\documentclass[12pt,reqno]{amsart}
 \usepackage{amssymb,amsmath, amsthm} 
  \usepackage[margin=1.2in,marginpar=1in]{geometry}

 \usepackage{comment}

\usepackage{tikz}
\usepackage{latexsym}
\usepackage{graphicx}
\usepackage{amssymb}
 \graphicspath{{Figures/}}
 \usepackage[show]{ed}
 \usepackage{soul}
  \usepackage{bm}
\usepackage[normalem]{ulem}

\renewcommand{\Re}{{\operatorname{Re}\,}}
\renewcommand{\Im}{{\operatorname{Im}\,}}

\renewcommand{\epsilon}{\varepsilon}

\newcommand{\R}{{\mathbb R}}

\newcommand{\lan}{\left\langle}
\newcommand{\ran}{\right\rangle}
\newcommand{\mc}[1]{\mathcal{#1}}

\newcommand{\supp}{{\operatorname{supp\,}}}
\renewcommand{\phi}{\varphi}

\newcommand{\ep}{\varepsilon}

\newtheorem{theo}{Theorem}

\newtheorem{lem}{Lemma}[section]
\newtheorem{prop}[lem]{Proposition}

\theoremstyle{definition}
\newtheorem{defn}[lem]{Definition}
\newtheorem{rem}[lem]{Remark}

\usepackage{hyperref}

\newcommand{\sub}[1]{_{_{\!#1}}}

\newcommand{\ds}{d\sigma\sub{H}}

\newcommand{\Uho}{\mc{U}\sub{H}(\tau_1)}
\newcommand{\Uhtau}{ \mc{U}\sub{H}(\tau)}

\newcommand{\chio}{\chi\sub{\,1}} 
\newcommand{\tchio}{\tilde\chi\sub{\,1}} 
\newcommand{\chit}{\chi\sub{\,2}} 
\newcommand{\tchit}{\tilde\chi\sub{\,2}} 

\title[Goodness bounds]{Lower bounds for eigenfunction restrictions in lacunary regions}

\author{Yaiza Canzani and John A. Toth}

\date{}

\begin{document}

\maketitle
\begin{abstract}  
Let $(M,g)$ be a compact, smooth Riemannian manifold  and $\{u_h\}$ be a sequence of $L^2$-normalized Laplace eigenfunctions 
that has a localized defect measure $\mu$ in the sense that $ M \setminus \supp\!(\pi_* \mu)   \neq \emptyset$ where $\pi:T^*M \to M$ is the canonical projection.
Using Carleman estimates  we prove that for any real-smooth closed hypersurface $H \subset (M\setminus \supp\!(\pi_* \mu))$ sufficiently close to  $ \supp\!(\pi_* \mu),$ and  for all $\delta >0,$
$$ \int_{H} |u_h|^2 d\sigma\sub{H} \geq C_{\delta} e^{-  [ \,  d(H, \supp\!(\pi_* \mu)) + \delta \, ] /h}$$
as $h \to 0^+$. We also show  that the result holds for eigenfunctions of Schr\"odinger operators and give applications to eigenfunctions on warped products and joint eigenfunctions of  quantum completely integrable (QCI) systems.

\end{abstract}

\section{Introduction}

Let $(M,g)$ be a compact, $C^{\infty}$ Riemannian manifold, with or without boundary.  Let $\Delta_g$ be the Laplace operator, 
and  consider $L^2$-normalized Laplace eigenfunctions $u_h \in C^{\infty}(M)$,
\begin{equation}\label{efn}
(-h^2\Delta_g - 1)u_h = 0, \qquad \qquad  \|u_h\|\sub{L^2}=1.
\end{equation}
In the case where $\partial M \neq \emptyset$,  we ask that the boundary be smooth and we impose that the $u_h$ satisfy either Dirichlet  or Neumann boundary conditions on $\partial \Omega$. That is, we ask that either  $u_h=0$ on $\partial M$ or  $\partial_\nu u_h=0$ on $\partial M$,  where $\partial_\nu$ is the normal derivative along the boundary.

Let $H \subset M$ be a smooth hypersurface.  Quantitative unique continuation for eigenfunction restrictions $ u_h |_H$ is an important property with applications to the study of eigenfunction nodal sets and has received a lot of attention in the literature over the past decade,  e.g., \cite{ DF88,  TZ09, BR15, ET15, CT18,   GRS13, Jun14, JZ16, TZ21}. Specifically, let $\{u_{h_j}\}_j$ be a sequence of eigenfunctions. Then, the question of whether there exist constants $C\sub{H} >0$ and $h_0>0$ such that
\begin{equation} \label{Goodness}
\int_H  |u_{h_j} |^2 \ds \geq e^{-C\sub{H}/h_j}
\end{equation}
for all $h_j\in (0,h_0]$ is, in general, an open question.  Here, $\ds$ denotes the measure on $H$ induced by the Riemannian metric. 

In principle, the validity of \eqref{Goodness} depends on {\em both} the particular eigenfunction sequence and the geometry of the hypersurface $H.$ In the terminology of  \cite{TZ09}, hypersurfaces $H$ for which \eqref{Goodness} is satisfied are said to be {\em good} for the eigenfunction sequence $\{u_{h_j}\}_{j=1}^\infty.$ It is a much more subtle problem than the analogue for submanifolds $U \subset M$ with $\dim U = \dim M$; indeed, in the latter case, the estimate $\| u_h \|\sub{L^2(U)} \geq e^{-C_U/h}$ follows by a well-known argument with Carleman estimates \cite{Zw}. For the hypersurface analogue in \eqref{Goodness},  there are comparatively few cases where (\ref{Goodness}) has been proved \cite{ BR15, ET15, CT18, GRS13,  Jun14}.  

 Our first result in Theorem \ref{mainthm1} deals with the case where the $u_{h}$'s have a defect measure, $\mu$, that is well-localized in the sense that $\supp\!(\pi_* \mu) \neq M$,
where 
$
\pi:T^*M \to M
$
 is the natural projection. See \cite[Chapter 5]{Zw} for background on defect measures. Roughly speaking, under this condition, our main result, Theorem \ref{mainthm1}, shows that \eqref{Goodness} is satisfied for a wide class of real-smooth separating hypersurfaces $H \subset M \backslash \supp\!(\pi_* \mu)$  that lie near $\supp(\pi_*\mu)$.

Before stating our first  result, we   introduce  the notion of a {\em lacunary region}. Throughout the paper, we write $\Uhtau$ for a Fermi collar neighbourhood of a hypersurface $H\subset M$ of width $\tau>0$ (see \eqref{fermi c}).  In the following, we work with $h$-pseudodifferential operators in $\Psi_{h}^{0}(\Uhtau)$ that are properly-supported in the Fermi tube $\Uhtau$ (see Subsection \ref{hpdo} for the definition).
In addition, given cutoffs $\chi_1, \chit \in C^{\infty}_0(\Uhtau),$ the notation $\chit \Subset \chio$ means that $\chio\equiv 1$ on $\supp \chit$

 \begin{defn}[\emph{$H$-lacunary region}] \label{weakdef}
 Let $M$ be a $C^\infty$ manifold, $H \subset M$  a closed $C^\infty$-hypersurface,  and $\tau>0$ such that $ \Uhtau$ is a Fermi collar neighborhood of $H$. 
 An $h$-pseudodifferential operator $Q(h) \in \Psi_{h}^{0}(\Uhtau)$ that is properly-supported in the Fermi tube $\Uhtau$ is said to be a \emph{lacunary operator}   for  the  sequence $\{u_h  \}$  if:\smallskip
 
\noindent (i) $Q(h)$ is  $h$-elliptic on $T^*\Uhtau$ (see Definition \ref{d:h elliptic}),\smallskip

\noindent (ii)  
for every  $\chio, \chit \in C^\infty_{0}(\Uhtau, [0,1])$   with $\chit \Subset \chio,$  there exists a constant $C >0$, depending only on the  sequence $\{ u_h \}$, such that
 \begin{equation} \label{lacunary}
 \| \chit Q(h) \chio  \,   \,u_{h} \|\sub{L^2(\Uhtau)}  = O(e^{-C/h}),  \qquad h\to 0^+. 
 \end{equation}
 
   We say that  there is an $H$-\emph{lacunary region} for the sequence $\{u_h  \}$  provided there exist $\tau>0$ and a lacunary operator  $Q(h) \in \Psi_{h}^{0}(\Uhtau)$ for $\{u_h  \}$. 
\end{defn}

  \begin{rem}\label{differential}
Since one can always replace $Q(h)$ with $Q(h)^* Q(h)$ in  \eqref{lacunary} above, without loss of generality, we also assume that the principal symbol $q_0$ is {\em real-valued}.   Also, we note that in the case where $Q(h)$ is $h$-differential, since it is then local, one can replace (\ref{lacunary}) with the condition that $ \| \chit Q(h) \, u_h \|_{L^2}   \leq C_0\,e^{-C/h} \| u_h \|_{L^2}.$
\end{rem}

\begin{rem}
We note that if the eigenfunction sequence $\{u_h\}$ has a defect measure $\mu$ associated to it, then  $Q(h) \in \Psi_{h}^{0}(\Uhtau)$ can only be a lacunary operator for $\{u_h\}$ provided 
$\supp\!(\pi_* \mu) \cap \Uhtau = \emptyset$.
\end{rem}

\begin{rem} 
The requirement that the order of the lacunary operator $Q(h)$ be zero is not necessary and is only made for convenience. We note that there is no loss of generality in making this assumption. For example, if $Q(h)$ is an $h$-elliptic differential operator of order $k \geq1$  and  $L(h)$ is a left-parametrix for $Q(h),$ then one can simply replace $Q(h)$ with  $L(h) Q(h)  \in \Psi_h^0(\Uhtau)$ as the new lacunary operator since by $L^2$-boundedness,
$$ \| L(h) Q(h) u_h \|_{L^2} = O(e^{-C/h}).$$
\end{rem}

In Section \ref{examples} we discuss the existence of lacunary operators in several examples. 
In the following, $d(H,K)$ denotes the Riemannian distance between the hypersurfaces $H$ and $K$.

The main result of the paper  is the following  theorem.

\begin{theo} \label{mainthm1}
Let $(M,g)$ be a compact $C^\infty$ Riemannian manifold.
Let $\{u_h\}$ be a sequence of eigenfunctions satisfying \eqref{efn} with an associated defect measure $\mu$ such that 
$$
K:= \supp\!(\pi_* \mu) \nsubseteq M.
$$
Let $H \subset K^c$  be a closed $C^\infty$-hypersurface or an open submanifold of a closed hypersurface and  suppose  that  $\Uhtau \subset K^c$  is an $H$-lacunary region for  $\{u_h\}$. Then, there exists $\tau_0>0$ such that 
 if  $0<d(H,K)<\tau_0$, it follows that for any $\ep >0,$ there are constants $C_{0}(\ep)>0$ and $h_0(\ep)>0$ with
$$ \| u_h \|\sub{ L^2({H})} \geq C_0(\ep)  e^{ - [d(H,K) + \ep] /h},$$
 for $h \in (0,h_0(\ep)].$
\end{theo}

The proof of Theorem \ref{mainthm1} involves two main ideas: First, in Theorem \ref{mainthm2} we prove a Carleman type estimate adapted to $H$ to obtain an exponential lower bound for $ \| u_h \|\sub{L^2(\mc{U}\sub{H}(\epsilon))}$.   This estimate shows that if  $\supp\!(\pi_* \mu)$ is close enough to $H$, then the positive mass detected by $\mu$ yields the lower bound  on the $L^2$-tubular mass  $ \| u_h \|\sub{L^2(\mc{U}\sub{H}(\epsilon))}$. Second, we show that the ellipticity of  $Q(h)$ allows us to factorize $Q(h)$ in the form $A(h)(hD_{x_n}-iB_0)$ where $A(h)$ is an $h$-elliptic operator, $B_0$ is a positive constant, and $x_n$ denotes the normal direction to $H$ (see Section \ref{factorization}). Then, a further refinement of the factorization argument (see Proposition \ref{key prop} for a precise statement) roughly speaking allows us to show that, since $Q(h)$ is lacunary for $\{u_h\}$, we can work as if  $(hD_{x_n}-iB_0)u_h = O(e^{-C/h})$ in the Fermi tube $\mc{U}\sub{H}(\epsilon)$.  We then use this observation together with a simple integration argument over the tube to obtain a lower bound on $ \| u_h \|\sub{L^2(H)}$ from the lower bound on $ \| u_h \|\sub{L^2(\mc{U}\sub{H}(\epsilon))}$.

The following Carleman estimate adapted to $H$ is crucial to the proof of Theorem \ref{mainthm1} and is of independent interest.

\begin{theo} \label{mainthm2}
Let $(M,g)$ be a compact $C^\infty$ Riemannian manifold.
Let $\{u_h\}$ be a sequence of eigenfunctions satisfying \eqref{efn} and let $\mu$ be a defect measure associated to it. 
Let $H \subset M$  be a closed  $C^\infty$-hypersurface (or any proper, open submanifold of a closed hypersurface) with $d(K,H) >0$, where 
$
K:= \supp\!(\pi_* \mu) \nsubseteq M.
$
Then, there exists $\tau_0>0$ such that if $d(H,K) < \tau_0,$ it follows that for any $\ep >0$
  there are constants  $C_0(\ep)>0$ and $h_0(\ep)>0$ with
  $$ \| u_h \|\sub{L^2(\mc{U}\sub{H}(\epsilon))}  \geq C_0(\ep) e^{ - ( d(H,K) + \ep )/h}$$
for all $h \in (0,h_0(\ep)].$ 
\end{theo}

\begin{rem} \label{local} We note that since we allow $H$ to have possibly non-empty boundary, that both Theorems \ref{mainthm1} and \ref{mainthm2} are {\em local} results. For example, when $\dim M = 2,$  Theorem \ref{mainthm1} gives exponential lower bounds for eigenfunction restrictions along lacunary $H$, where $H$ can be a curve segment of  arbitrarily small length fixed independent of $h$.
\end{rem}

Our  results also hold for eigenfunctions of a Schr\"odinger operator $P(h) = -h^2 \Delta_g + V - E$  where $V \in C^{\infty}(M;\R)$ is a  real smooth potential and $E$ is a regular value for it:
\begin{equation}\label{efnSchr}
(-h^2 \Delta_g + V - E)u_h = 0, \qquad \qquad  \|u_h\|\sub{L^2}=1.
\end{equation}
 We indicate the fairly minor changes to the proofs  in Section \ref{schrodinger}.

\begin{theo}\label{mainthm3}
Let $(M,g)$ be a compact $C^\infty$ Riemannian manifold.
Let $\{u_h\}$ be a sequence of eigenfunctions satisfying \eqref{efnSchr} with an associated defect measure $\mu$ such that 
$
K:= \supp\!(\pi_* \mu) \nsubseteq M.
$
Let  $H \subset K^c$  be a lacunary $C^\infty$-hypersurface.  Then, there exists  $\tau_0>0$ such that if 
$
0<d(H,K)< \tau_0,
$
then for all  $\ep>0$ there are constants $C_0(\ep) >0$ and $h_0(\ep)>0$ such that for all $ \beta > \max_{x\in M} | V(x) - E|^{1/2}$, 

$$ \| u_h \|\sub{L^2(H)} \geq C_0(\ep) e^{ -    \beta \,  [d(H,K) + \ep]/h},$$
for all  $h \in (0,h_0(\ep)].$ 
\end{theo}

\subsection{Outline of the paper.} The Carleman estimates required for the proof of Theorem  \ref{mainthm2} are proved in Section 2.  In Section \ref{goodness} we combine the result in Theorem \ref{mainthm2} with the operator factorization argument in Proposition \ref{key prop}  to prove Theorem \ref{mainthm1}.   In Section \ref{schrodinger}, we indicate the relatively minor changes needed to handle the case of Schr\"{o}dinger operators. Finally, in Section \ref{examples} we present several examples to which our results apply.

\subsection{Acknowledgements.}  Y.C. was supported by the Alfred P. Sloan Foundation, NSF CAREER Grant DMS-2045494, and NSF Grant DMS-1900519. J.T. was
partially supported by NSERC Discovery Grant \# OGP0170280 and by the French National Research Agency project Gerasic-ANR- 13-BS01-0007-0.

\section{Carleman estimates: Proof of Theorem \ref{mainthm2}} \label{carleman}

This section is devoted to the proof of Theorem \ref{mainthm2}.  In Section  \ref{weight function} we construct  a Carleman weight and in Section \ref{regions} we introduce the relevant regions we will use in the proof of Theorem \ref{mainthm2} to infer the lower bound on the eigenfunction $L^2$-mass near $H$ using the assumption on the support of the defect measure $\mu$.  The actual proof of Theorem \ref{mainthm2} is given in Section \ref{proof of 2}.

\subsection{Carleman weight} \label{weight function}

 Given a hypersurface $H \subset K^c,$ using the fact that $K$ is closed, we choose $q_0 \in K$ and a corresponding point $q\sub{H} \in H$ with the property that
 $$ d(q_0, q\sub{H}) = d(K,H) >0.$$

Given  $q_0 \in K$ we let $ Y = Y_{q_0} \subset M$ be a geodesic sphere containing $q_0$ with center $q\sub{H},$ so that the line segment $\overline{q_0 q\sub{H}}$ is a radial geodesic  (see Figure 1).

Let $(y', y_n)$ be geodesic normal coordinates adapted to $ Y,$  so that
\begin{equation}\label{Y coord}
 Y=\{y_n=0\}, \qquad q_0=(0,0),
\end{equation}
and with $\{y_n>0\}$ corresponding to points in the interior of the ball whose boundary is $ Y$ (see \cite[Corollary C.5.3]{HOV3}).  Let $\tau\sub{ Y}>0$ be  chosen so that the   $y_n$-coordinates  are defined for $- 2 \tau\sub{ Y} \leq y_n \leq 2 \tau\sub{ Y}$ and let  $c\sub{ Y}>0$, be chosen so  that the remaining tangential coordinates $y'$ along $Y$ are defined for  $|y'|< c\sub{ Y}$ .  We also fix an arbitrarily small constant $\ep\sub{ Y}>0$ which will be specified later on (see (\ref{constconditions})).

For $0<\tau\leq \tau\sub{ Y}$ and $0<\ep\leq \ep\sub{Y}$  we  will carry out a Carleman argument in the  rectangular domain (see Figure 1)

\begin{equation}\label{W}
\mathcal{W}\sub{ Y}(\tau,\ep):=\{(y', y_n):\; |y'|<c\sub{ Y},\;\; -2\ep<y_n<\tau+2\ep\},
\end{equation}\

We also introduce a    tangential cutoff $\rho_\ep=\rho_\ep(y')\in C^{\infty}_{0}( \{ |y'| < c\sub{Y} \})$ along the geodesic sphere $Y$ satisfying 
\begin{enumerate}
\item $\supp \, \rho_\ep  \subset \{|y'| > 3\ep \},$
\item $ \rho_\ep \leq  0$,
\item $\rho_\ep(y) \equiv -1$ on $\{ \tfrac{1}{3}c\sub{ Y}< |y'|   <  c\sub{ Y}\}$,
\item there exists $c>0$, independent of $\ep$, such that 
 \begin{equation}\label{rho'}
 | \partial_{y'} \rho_\ep(y')|\leq c, \qquad  |y'|<c\sub{Y}, \qquad  0<\ep<\ep\sub{ Y}.
 \end{equation}
 \end{enumerate}

 For $0<\tau<\tau\sub{Y}$ and $0<\ep\leq \ep\sub{Y}$,  we consider the putative  weight function,  $\psi=\psi_{\ep,\tau}$,

\begin{equation} \label{weight}
\psi \in C^{\infty}(\mc{W}\sub{ Y}(\tau, \ep\sub{Y})), \qquad \psi(y',y_n) := y_{n}  +  2\tau  \rho_\ep(y'),
 \end{equation}\
 
and form the conjugated operator 
\begin{equation}\label{P}
P_{\psi}(h):= e^{\psi/h} P(h) e^{-\psi/h} : C^{\infty}_{0}(\mc{W}\sub{ Y}(\tau,  \ep\sub{Y})) \to C^{\infty}_0(\mc{W}\sub{ Y}(\tau,  \ep\sub{Y}))
\end{equation}
with principal symbol
\begin{equation}\label{ppsi}
p_{\psi}(y,\xi) := p(y, \xi + i \partial_{y} \psi),
\end{equation}
where $p(y,\xi) = |\xi|_{g(y)}^2 -1.$  We note for future reference that from (\ref{weight}),  the weight function $\psi$ implicitly depends on the parameters $\ep$ and $\tau.$

\begin{lem}\label{weight lemma}
There exists $ \tau\sub{ Y}>0$, depending only on the principal curvatures of $Y$,  such that for all $0<\tau\leq \tau\sub{Y}$ and all  $0<\ep <\ep\sub{ Y}$  the function $\psi \in  C^{\infty}(\mc{W}\sub{ Y}(\tau, \ep\sub{Y}))$  in (\ref{weight}) is a Carleman weight with
\begin{equation} \label{carlwt}
\{ \Re p_{\psi}, \Im p_{\psi} \} >0 \qquad \text{on}\;\; \{p_{\psi}=0\}.
 \end{equation}
\end{lem}

\begin{proof}
To prove the claim in \eqref{carlwt} note that (see e.g. \cite[Lemma 2.1]{SjZw}) the principal symbol of $P(h)=-h^2\Delta_g-I$ in the $(y', y_n)$ coordinates, with dual variables $(\xi', \xi_n)$,  takes the form
 \begin{equation}\label{ppal-symbol}
 p(y, \xi)=\xi_n^2+a(y', \xi') -2y_n b(y', \xi') + R(y, \xi') -1,
 \end{equation}
with $R(y, \xi')=O(y_n^2|\xi'|^2)$.
The expansion in (\ref{ppal-symbol}) holds near $ Y$  for $y_n$ small. Here, $a$ is the quadratic form dual to the induced metric on $ Y$ and $b$ is the quadratic form dual to the second fundamental form for $ Y$. In particular, since $ Y$ is strictly convex, and the eigenvalues of $b(y', \xi')$ with respect to $a(y', \xi')$ are the principal curvatures of $ Y$, the eigenvalues are strictly positive.

 Note that by \eqref{ppal-symbol}, with $p_\psi$ as in \eqref{ppsi}, we have 
  $$
  \Im p_{\psi} = 2\xi_n +  O(\tau\sub{Y}),  \qquad
 \Re p_{\psi}=p(x, \xi)-1 +O(\tau\sub{Y}^2),
 $$
and so, $\Re p_{\psi}=a(y', \xi')-2+O(\tau\sub{Y})$ on $p_{\psi}=0.$ Consequently,
\begin{align} \label{char}
 \{p_{\psi}=0\}
&=  \big\{ (y,\xi) \in T^* (\mc{W}\sub{ Y}(\tau, \ep\sub{Y})):\;  a(y', \xi') = 2 + O(\tau\sub{Y}), \; \;\;\xi_n = O(\tau\sub{Y})\big\}.
\end{align}
A direct computation shows that there is $C\sub{Y}>0$ such that for $(y,\xi) \in \{p_{\psi}=0\},$
\begin{align} \label{subelliptic}
 \{ \Re p_{\psi}, \Im p_{\psi} \} (y,\xi) 
&=  4b(y', \xi')   + O(\tau\sub{Y})\\
 & \geq C\sub{Y} a(y',\xi') + O(\tau\sub{Y}) 
 \geq 2 C\sub{Y} + O(\tau\sub{Y}). \notag
\end{align}

In (\ref{subelliptic}) we have used  that the eigenvalues of $b(y',\xi')$ are strictly positive and that we are restricting $b$ to co-vectors with norm bounded below since $a(y', \xi') = 2 + O(\tau\sub{Y})$ on $\{p_{\psi}=0\}$.
Note that  the $O(\tau\sub{Y})$ errors above depend only on the principal curvatures of $ Y$. Therefore, there is  a sufficiently small $\tau\sub{Y} >0$,  depending only on the curvature of $ Y$, such that the RHS of (\ref{subelliptic}) is positive and consequently, \eqref{carlwt} holds for $0<\tau<\tau\sub{Y}$. 
\end{proof}

For concreteness, in the following we will  assume that with $\tau\sub{Y} >0$ as in Lemma \ref{weight lemma}, 
\begin{equation} \label{constconditions}
 \ep\sub{ Y} <\min(\tfrac{1}{10} \tau\sub{Y}, \tfrac{1}{10}c\sub{ Y}).
  \end{equation}

\subsection{Control, transition, and black-box regions} \label{regions}
Given $H \subset K^c$ a smooth hypersurface, we choose points $q_0 \in K=\supp\!(\pi_* \mu)$ and $q\sub{H} \in H$ as in Subsection \ref{weight function} and let $\tau\sub{Y}>0$ be as in Lemma \ref{weight lemma}. In the following, we also assume that $\tau\sub{Y}<r_{q_0},$ where $r_{q_0}>0$ is the maximal radius  for which the exponential map $\exp_{q_0}$ is a diffeomorphism on $B(q_0, 2r_{q_0}),$ and set
\begin{equation}\label{distance}
\tau\sub{H}:=d(q_0, q\sub{H})= {d(K,H),}  \qquad 0 < \tau\sub{H} < \tau\sub{Y}<  r_{q_0}.
\end{equation}

In addition, by possibly shrinking $\ep\sub{Y}>0$ further in (\ref{constconditions}), we assume from now on that
\begin{equation} \label{ep}
 \tau\sub{H} + 2 \ep\sub{Y} <\tau\sub{Y}. 
 \end{equation}

Let $\gamma$ be a unit speed geodesic joining $q_0=\gamma(0)$ with $q\sub{H}=\gamma(\tau\sub{H})$.  To fix the reference geodesic sphere $Y$ once and for all,  we let $q=\gamma(r_{q_0})$  and set 

\begin{equation}\label{Sigma def}
 Y:=\partial B(q, r_{q_0}).
\end{equation}

As in \eqref{Y coord}, we continue to work in geodesic normal coordinates   $(y', y_n): \mc{U}\sub{Y}(\tau\sub{Y}) \to \R^n$ adapted to $ Y$ with
\begin{equation}\label{fermi c}
 Y=\{y_n=0\}, \qquad q_0=(0,0),
\end{equation}
 and with $\{y_n>0\}$ corresponding to points in the interior of $B(q, r_{q_0})$.
In these coordinates,  $\mc{U}\sub{Y}(\tau\sub{Y})=\{(y', y_n):\; |y_n|<\tau\sub{Y}\}$.

 In the following, we assume that $\tau\sub{ H}$,  $\ep\sub{ Y}$, and  $c\sub{ Y}$  satisfy \eqref{constconditions},   (\ref{ep}) and
assume
\begin{equation}\label{epsilon}
 0<\ep < \min(  \ep\sub{Y}, \tfrac{1}{10}\tau\sub{H}).
\end{equation}

We  carry out the Carleman argument in the rectangular domain $\mc{W}\sub{ Y}(\tau\sub{H}, \ep)$ defined in \eqref{W}, where $\tau\sub{H} = d(q_0,H).$ Within this set, we identify three key regions: \, the  control region $ U_{\!cn}(\epsilon)$,  \, the  transition region $U_{tr}(\epsilon)$, and the black-box region  $U_{bb}(\epsilon)$. Here, $U_{\!cn}(\epsilon)$ refers to an $\epsilon$-tube near $ Y$,  $U_{bb}(\epsilon)$ is the region where we wish to prove lower bounds, and $U_{tr}(\epsilon)$ are the {transitional} regions connecting the two former regions (see Figure 1).   To define these we need the following cut-off functions.

Let  $\ep>0$ be a small constant satisfying the {bound in \eqref{epsilon}.} We define $  \chi_{\epsilon,_Y} \in C^{\infty}(\R;[0,1])$ with 
 $$
 \begin{cases}
  \chi_{\epsilon,_Y}(y_n) = 1& \;\; y_n > - \tfrac{1}{2}\epsilon,\\
  \chi_{\epsilon,_Y}(y_n) = 0& \;\; y_n < - 2 \epsilon,
 \end{cases}  \qquad \qquad \supp \, \partial   \chi_{\epsilon,_Y}  \subset \{ -2\epsilon < y_n < -\epsilon \}.
 $$ 
  Let  $  \chi_{\epsilon,_H} \in C^{\infty}(\R;[0,1])$ be a cutoff localized around  $\{y_n =\tau\sub{H}\}$ with
  $$ 
\begin{cases}
  \chi_{\epsilon,_H}(y_n) = 0&\quad  y_n > \tau\sub{H}  + 2 \epsilon,\\
  \chi_{\epsilon,_H}(y_n) = 1& \quad y_n < \tau\sub{H} - 2 \epsilon,
\end{cases}  \qquad \qquad  
\supp \, \partial   \chi_{\epsilon,_H}  \subset \{ | y_n - \tau\sub{H}|<\epsilon \}.
$$
Let $\chi_{\epsilon,tr} \in C^{\infty}_0(\R^{n-1};[0,1])$ be a {transitional} cutoff with
$$
\begin{cases} 
  \chi_{\epsilon,tr}(y') = 0& \quad  |y'| >c\sub{ Y},\\
  \chi_{\epsilon,tr}(y') = 1& \quad |y'|  <  4\ep,
\end{cases}
 \qquad \qquad 
 \supp \, \partial   \chi_{\epsilon,tr}  \subset \{ \tfrac{1}{3}c\sub{ Y}< |y'|   < c\sub{ Y} \}.
$$

\begin{figure}
  \caption{}
\includegraphics[width=15cm]{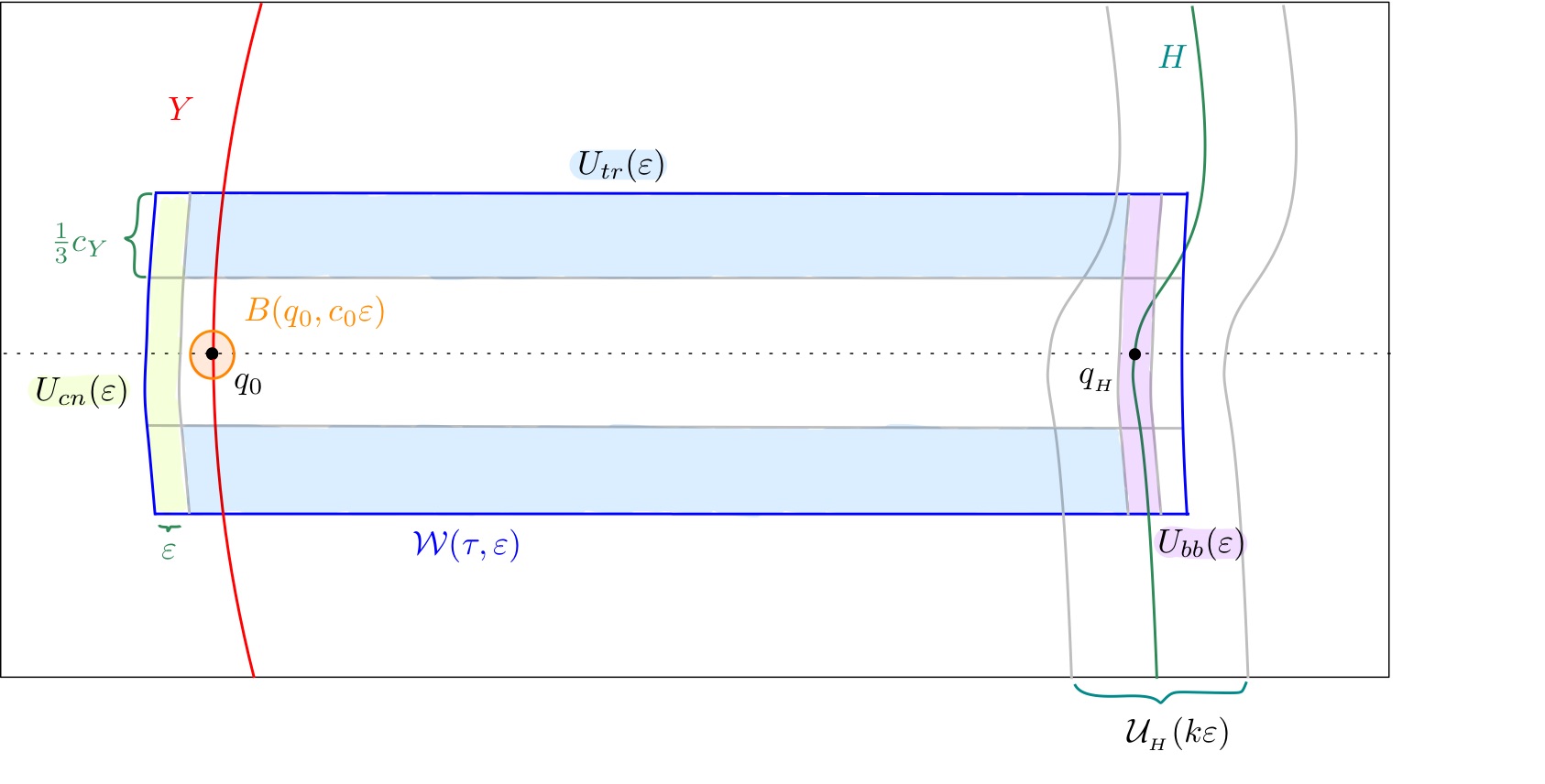}
\end{figure}

Finally, we define the cutoff function  $\chi_{\ep} \in C^{\infty}_{0}(\mc{W}\sub{ Y}(\tau\sub{H},\ep\sub{Y}))$ with 
\begin{equation} \label{cutoff}
\chi_{\epsilon}(y',y_n): =   \chi_{\epsilon,_Y}(y_n) \cdot   \chi_{\epsilon,_H}(y_n) \cdot   \chi_{\epsilon,tr}(y').
\end{equation}
By the Leibniz rule it follows that 
\begin{equation}\label{leib0}
\supp \,  \partial \chi_{\epsilon} \subset  U_{\!cn}(\epsilon) \cup U_{bb}(\epsilon) \cup U_{tr}(\epsilon),
\end{equation}
where, as shown in Figure 1,
\begin{align*}
 U_{\!cn}(\epsilon)&:= \supp \, \partial   \chi_{\epsilon,_Y}  \times  \supp   \chi_{\epsilon,tr}
 \subset \{-2\epsilon < y_n < -\epsilon, \; |y'|<c\sub{ Y}\}\\
 U_{bb}(\epsilon)&:= \supp \, \partial   \chi_{\epsilon,_H}  \times \supp \,  \chi_{\epsilon,tr}
 \subset\{ | y_n - \tau\sub{H}|<\epsilon,\; |y'|<c\sub{ Y}\}\\
 U_{tr}(\epsilon)&:= \supp (  \chi_{\epsilon,_Y}\chi_{\ep, _H})\times \supp \, \partial   \chi_{\epsilon,tr}
 \subset \{ -2\ep<y_n<\tau\sub{H}+\ep, \;  \tfrac{1}{3}c\sub{ Y}< |y'|   < c\sub{ Y}\}.
\end{align*}

 We note that one can refine the containment in \eqref{leib0} slightly by setting
 $$
 \tilde{U}_{tr}(\epsilon): = U_{tr}(\epsilon) \setminus  \big( U_{bb}(\epsilon) \cup U_{\!cn}(\epsilon) \big),
 $$
  and noting that Leibniz rule actually gives
 \begin{equation} \label{leib}
 \supp \,  \partial \chi_{\epsilon}\subset \tilde{U}_{tr}(\epsilon) \cup U_{bb}(\epsilon) \cup U_{\!cn}(\epsilon). 
 \end{equation}

\subsection{Proof of Theorem \ref{mainthm2}} \label{proof of 2}
Let $q_0 \in  \supp(\pi_*\mu)  $, $\tau\sub{Y}>0$ be as in Lemma \ref{weight lemma}, and $q\sub{H} \in H$ be chosen so that
$\tau\sub{H}=d(q_0, q\sub{H})= d(K, H).$ We continue to let $ Y$ be  the geodesic sphere defined in \eqref{Sigma def} with $(y',y_n): \mc{U}\sub{Y}(\tau\sub{Y}) \to  \R^n$ geodesic normal coordinates adapted to $ Y$ as in \eqref{fermi c}.
For $\ep\sub{ Y}$ satisfying (\ref{constconditions})
let
$$
 \mc{W}:=\mc{W}\sub{ Y}(\tau\sub{H}, \ep\sub{Y}).
 $$

 After possibly shrinking $\ep\sub{Y}>0$ further, depending only on $(H,Y)$, there exists $k>2$ such
that for  $0<\ep<\ep\sub{ Y}$ (see Fig 1)
 \begin{equation}\label{Ubb is nice}
  \{ (y',y_n):\; \,\, |y_n- \tau\sub{H}|< \epsilon, \,\,  |y'|  \leq 4\epsilon \}
 \subset \mc{U}\sub{H}(k\epsilon).
\end{equation}

We also choose $c_0>0$ so that for $0<\ep<\ep\sub{Y}$ (see Fig 1), the {\em control ball}
  \begin{equation}\label{small ball}
 B(q_0, c_0\epsilon)\subset \{(y', y_n): \; |(y',y_n)| < \tfrac{1}{5}\epsilon\},
 \end{equation}
and continue to assume that \eqref{ep} holds; that is, $ \tau\sub{H} + 2 \ep\sub{Y} <\tau\sub{Y}$.

We now carry out the Carleman argument.  With $\ep$ as in \eqref{epsilon} and
  $\chi_{\ep} \in C^{\infty}_{0}(\mc{W})$ as defined in \eqref{cutoff}, set
\begin{equation}\label{v_h}
 v_h:= \chi_{\epsilon} e^{\psi/h} u_h,  \qquad \psi(y', y_n)=y_{n}  +  2\tau\sub{H} \rho_\ep(y'),
 \end{equation}
 with $\psi$ as in  \eqref{weight} with $\tau\sub{H}$ in place of $\tau$.
 By Lemma \ref{weight}, since $0<\tau\sub{H}<\tau\sub{Y},$
 $ \psi \in  C^{\infty}(\mc{W})$ is a Carleman weight. Thus,
   by the subelliptic Carleman estimates \cite[Theorem 7.5]{Zw}, there exists $C>0$ so that, with $P_\psi(h)$ as in \eqref{P},

\begin{equation} \label{carleman}
\| P_{\psi}(h) v_h \|^2\sub{L^2(\mc{W}) } \geq C h \, \| v_h \|\sub{L^2(\mc{W})}^2.
 \end{equation}\

Note that   $\chi_\ep=1$  on $B(q_0, c_0\ep)$  by \eqref{small ball} and $B(q_0, c_0\ep) \subset \mc{W}.$
In addition, since $\rho_\ep=0$ on $B(q_0, c_0\ep)$, we have 
\begin{equation} \label{rhs weight}
\psi(y) = y_n \geq -\tfrac{1}{2}\ep, \qquad (y',y_n) \in B(q_0, c_0\ep).
\end{equation}
  Also, since  $q_0 \in \supp\!(\pi_* \mu),$ it follows that  for all $r>0$ there is $C(r)>0$ such that
  $$ \pi_* \mu ( B(q_0,r) ) \geq C(r)>0.$$
  In particular, there exist constants $C(\epsilon)>0$ and $h_0(\epsilon)>0$ such that for $h \in (0,h_0(\epsilon)],$
  
  \begin{equation} \label{masslower}
   \int_{B(q_0,c_0\epsilon)} |u_h|^2 \, dv_g \geq C(\epsilon). 
   \end{equation} 
   Thus,  from   \eqref{v_h}, \eqref{rhs weight}  and \eqref{masslower} it follows  that there exist $C(\ep)>0$ and $h_0(\ep)>0$ such that
 \begin{eqnarray} \label{RHScontrol}
  \|v_h \|_{L^2(\mc{W})}^2 
  \geq e^{-\epsilon/h}  \int_{B(q_0,c_0\epsilon)} |u_h|^2 \, dv_g  \geq C(\epsilon)  e^{-\epsilon/h}, 
  \end{eqnarray}
  for $h \in (0,h_0(\epsilon)].$  Here, (\ref{RHScontrol}) gives the required lower bound for the RHS in (\ref{carleman}).

Next, since $P(h) u_h = 0,$ we will use that 
\begin{equation}\label{P-com}
P_{\psi}(h) v_h = e^{\psi/h} [ P(h), \chi_{\epsilon}] u_h.
\end{equation}
Also, since $[P(h), \chi_{\epsilon}]$ is an $h$-differential operator of order one supported in 
$
\supp \partial \chi_{\epsilon} \subset \tilde{U}_{tr}(\epsilon) \cup U_{bb}(\epsilon) \cup U_{\!cn}(\epsilon),
$
 where the inclusion was derived in \eqref{leib}. 
Thus,  from \eqref{carleman} and (\ref{RHScontrol}) it follows that, after possibly shrinking   $C(\ep),$
 \begin{eqnarray} \label{carleman4}
\| P_{\psi}(h) v_h \|^2_{L^2( U_{bb}(\epsilon) ) } +  \|P_{\psi}(h) v_h \|^2_{L^2( U_{\!cn}(\epsilon) ) }  
 + \| P_{\psi}(h) v_h \|^2_{ L^2(\tilde{U}_{tr}(\epsilon) )}    \geq C(\ep) h e^{-\epsilon/h}.
 \end{eqnarray} 
 
   We proceed to find upper bounds for each term in the LHS of \eqref{carleman4}. On the control set $ U_{\!cn}(\epsilon)$ we have that  $-2\epsilon < y_n < -\epsilon$ and so, since $\rho_\ep \leq 0,$ 

\begin{equation} \label{conweight}
 \psi(y) \leq y_n, \qquad y \in U_{\!cn}(\ep). \end{equation}

 From \eqref{P-com} and \eqref{conweight},  it follows by $L^2$-boundedness that there are constants  $\tilde C>0$  and $\tilde h_0>0$ such that 
\begin{equation}\label{control bound}
 \|P_{\psi}(h) v_h \|^2_{L^2( U_{\!cn}(\epsilon) ) }  \leq \| e^{y_n/h} [P(h),\chi_{\epsilon}] u_h\|^2_{L^2( U_{\!cn}(\epsilon))} \leq \tilde C h^2 e^{-2\epsilon/h},
\end{equation}
for all $0<h<\tilde h_0$.

On the transition set  $\tilde{U}_{tr}(\epsilon)$ we have $\rho_\ep(y') = -1$ and so from \eqref{epsilon} it follows that 
$$ 
\psi(y) = y_n + 2\tau\sub{H} \rho_\ep(y') = y_n - 2 \tau\sub{H}  \leq \tau\sub{H}  + \ep- 2\tau\sub{H}   \leq  - \tfrac{9}{10} \tau\sub{H} < - 9 \ep, \, \qquad y \in U_{tr}(\ep). $$

Therefore,   after possibly adjusting $\tilde C$ and $\tilde h_0$, and recalling (\ref{constconditions}) and \eqref{P-com},
\begin{equation} \label{transition bound}
 \|P_{\psi}(h) v_h \|^2_{L^2( U_{tr}(\epsilon) ) }  \leq \tilde C h^2 e^{-{18} \epsilon/h}, 
 \end{equation}
for all $0<h<\tilde h_0(\ep)$.

In view of \eqref{control bound} and  \eqref{transition bound},  both the transition and control terms on the LHS of \eqref{carleman4} can be absorbed into the RHS for $h>0$ small. The result is that there are constants $C(\ep)>0$ and  $h_0(\ep)>0$ such that for all $0<h<h_0(\ep)$
  \begin{eqnarray} \label{carleman5}
 \| e^{\psi/h} [ P(h), \chi_{\epsilon}] u_h \|^2_{ L^2( U_{bb}(\epsilon)  ) }=\| P_{\psi}(h) v_h \|^2_{L^2( U_{bb}(\epsilon) ) }   \geq C(\ep) h e^{-\epsilon/h}. 
 \end{eqnarray}
 
 Next, on the black-box set  ${U}_{bb}(\epsilon)$  we have $y_n<\tau\sub{H}+\ep$ and so,
 $$
 \psi(y) \leq \tau\sub{H}+\ep, \quad y \in {U}_{bb}(\epsilon),
 $$ 
 since $\rho_\ep \leq 0.$  So,  \eqref{carleman5} implies that
  \begin{equation}\label{almost}
\tilde{C}h^2\|u_h\|_{ L^2( U_{bb}(\epsilon)  ) } ^2 \geq  \| [ P(h), \chi_{\epsilon}] u_h \|^2_{ L^2( U_{bb}(\epsilon)  ) }    \geq Ch e^{-(2\tau\sub{H} +3\epsilon)/h},
  \end{equation}
for some $\tilde{C}>0$.
 
From  \eqref{Ubb is nice}, it is clear that for all $0<\ep<\ep\sub{Y}$ 
\begin{align*}
 U_{bb}(\epsilon)
  \subset \{(y',y_n):  \,\, |y_n- \tau\sub{H}|< \epsilon, \,\,  |y'|  \leq 3\epsilon \}
 \subset \mc{U}\sub{H}(k\epsilon).
\end{align*}
Then, by \eqref{almost}, there are $C(\ep)>0$ and  $h_0(\ep)>0$  such that
  $$
\|u_h\|_{ L^2 ( \mc{U}\sub{H}(k\epsilon) ) } ^2   \geq C(\ep) h^{-1} e^{-(2\tau\sub{H} +3\epsilon)/h}, \qquad 0 < h<h_0.
  $$

 Since by assumption  $k>2,$ we have $\tfrac{3}{2}\ep<k\ep$ and by setting $\tilde \ep=k\ep$, it follows that 
$\|u_h\|_{ L^2( \mc{U}\sub{H}( {\tilde\epsilon} ) )}    \geq Ch^{-{1}/{2}} e^{-(\tau\sub{H} +\tilde \ep)/h}$ for all $0<\tilde \ep< k\ep_0$ and $0<h<h_0$. The theorem then  follows since $\tau\sub{H}=d(q_0,H)$.
\qed\\

\section{Goodness estimates in lacunary regions: Proof of Theorem \ref{mainthm1}}\label{goodness}
In this section, we prove Theorem \ref{mainthm1}. Before carrying out the proof, we briefly {recall} some background material.




\subsection{Semiclassical pseudodifferential operators (h-pseudos)} \label{hpdo}

Let $U\subset M$ be open. We say that $a \in S^{m}_{h}(U)$  provided $a \sim h^{-m} (a_0 + h a_1 + \dots)$ in the sense that for all $\ell \geq 0$
\begin{equation}
\label{scsymbol} 
\begin{gathered} 
 a - h^{-m} \sum_{0 \leq j \leq \ell} h^{j} a_j \;\; \in h^{-m+ \ell+1} S^{0}(U),
 \end{gathered}
 \end{equation}
 where (see \cite[Section 14.2.2]{Zw})  
 $$
  S^{0}(U) =\Big \{ a \in C^{\infty}(T^*U):\;  \partial_{x}^\alpha \partial_{\xi}^{\beta} a(x,\xi) = O_{\alpha,\beta}(1) \; \text{for all }\, \alpha, \beta \in \mathbb N^n,\,\, (x,\xi) \in T^*U \Big\}.
  $$
Consider now the special case where $H \subset M$ is a closed hypersurface and $\Uhtau$ is an open Fermi tube about H of width $\tau>0.$ 
 In the following, we let $ x =(x',x_n): \Uhtau \to \R^n$ be Fermi coordinates centered on the hypersurface $H= \{x_n = 0 \}.$  We say that that $P(h)$ is an $h$-pseudodifferential operator ($h$-pseudo) on the tube  $\Uhtau$ if its kernel can be written in the form 
$$
P(x,y;h)=K_a(x,y;h)+R(x,y;h)
$$ 
where
\begin{equation}\label{kernel}
K_a(x,y;h)=\frac{1}{(2\pi h)^m}\int_{\R^n} e^{\frac{i}{h}\lan x-y,\xi\ran} \tchio(x_n) \, a(x,\xi;h) \, \tchit(y_n) d\xi,
\end{equation}
and for  all $\alpha,\beta  \in \mathbb N^n $,
$$
|\partial_x^\alpha \partial_y^\beta R(x,y)| = {\mathcal O}_{\alpha,\beta}(h^{\infty}).
$$
Here, $\tchio, \tchit \in C^{\infty}_0(\Uhtau)$ are  tubular cutoffs with $\tchio \Subset \tchit$ and $a\in S^{m}_{h}(\Uhtau)$.   As for the corresponding operator, we write $P(h) \in \Psi^{m}_{h}(\Uhtau).$

In the following it will also be useful to introduce two other cutoffs $\chio, \chit \in C^{\infty}_0(\Uhtau)$ with 
$\chit \Subset \tilde{\chi_2} \Subset \chio\Subset \tilde\chio.$  For concreteness, choosing $\ep \in (0, \ep\sub{Y})$ we assume that
\begin{itemize} 
\item $\tchio \in C^{\infty}_0([-2\ep,2\ep])$ with $\chio(x_n) = 1 $ when $|x_n| \leq \ep$, 
\item $\chio \in C^{\infty}_0([-\ep,\ep])$ with $\chio(x_n) = 1 $ when $|x_n| \leq \ep/2$, 
\item $\tchit \in C^{\infty}_0([-\ep/2,\ep/2])$ with $\tchio(x_n) = 1$ when $|x_n| \leq \ep/4,$
\item  $\chit \in C^{\infty}_0( (-\ep/4, \ep/4) )$ with $\chit(x_n) = 1$ when $|x_n| \leq \ep/8$.
\end{itemize}

 For  convenience, in the following we will use Fermi coordinates $x = (x',x_n)$, with $H=\{x_n=0\}$, 
 to represent the $h$-pseudo's without further comment. 
Given a symbol $a \in S^{m}_{h}(\Uhtau)$ we write  $a(x,hD_x)$ for the operator whose kernel is given by \eqref{kernel}.
In what follows, we will say that  $a \in S^{m}_h({\Uhtau})$ is  a \emph{tangential symbol} if $a=a(x,\xi')$ does not depend on the geodesic conormal variable $\xi_n$. In this case we also say that $a(x,hD_{x'})$ is a tangential operator.

For future reference, we also recall the following basic definition.
\begin{defn}\label{d:h elliptic}
We say that $A\in \Psi_{h}^{m}(\Uhtau)$ is \emph{$h$-elliptic} if there exists $C_0>0$ such that  the principal symbol $\sigma(A) := a_0$ satisfies
\begin{equation*}
|a_0(x,\xi)|\geq C_0 h^{-m}, \qquad (x,\xi)\in T^*(\Uhtau).
\end{equation*}
\end{defn}

For more detail on the calculus of  h-pseudos, we refer the reader to \cite{Zw,MarBook}.

\subsection{Operator factorization}\label{factorization}

In this section  we carry out a  factorization of an $h$-elliptic pseudo $Q(h) \in \Psi_h^0(\Uhtau)$  over a Fermi tube $\Uhtau$ in terms of the diffusion operator $hD_{x_n} - i B(x,hD_{x'})$, where  $B(x,\xi') \gtrapprox 1$ is a tangential symbol.

\begin{prop}\label{F and G}
 Let  $\tau \in (0, \tau\sub{H}]$,   $Q(h)\in \Psi_{h}^{0}(\Uhtau)$ be an $h$-elliptic operator  and  $\chio, \chit \in C^{\infty}_0(\Uhtau, [0,1])$  be the tubular cutoffs defined above with  $\chio \Subset \chit.$ 
Let $B \in S^{0}({\Uhtau})$ be a real valued tangential symbol and  $c_0>0$ such that   $B(x,\xi')  \geq  c_0$  when $(x,\xi) \in T^*\Uhtau$.

Then,   there exists an $h$-elliptic operator $ A(h) \in \Psi_{h}^{0}(\Uhtau)$ such that  
  \begin{equation} \label{OPfact}
\big\|\chit  \,   \Big( Q(h) - A(h) \, \big(  hD_{x_n}  - i B(x,hD_{x'}) \big) \,  \Big) \, \chio   \big\|\sub{L^2 \to L^2}
=  {{\mathcal O}(h^{\infty})}. 
\end{equation} 
\end{prop}
\begin{proof}
Let $q \sim \sum_{j=0}^{\infty} q_{-j} h^{j}$ be the symbol of $Q(h).$ 
First, note that
$$ 
|\xi _n - i  B(x,\xi')|^2 = \xi_n^2 + | B(x,\xi')|^2 \geq  c_0^2,  \quad (x,\xi)  \in T^*\Uhtau, 
$$
 and since by assumption $Q(h)$ is $h$-elliptic, we have that
%
%
$$a_0(x,\xi): = \frac{ q_0(x,\xi)}{ \xi_n - i   B(x,\xi') } \in S^{0}(T^*\Uhtau),\qquad  a_0(x,\xi) \geq C>0.$$
Our objective is to obtain an operator factorization   of the  form
$$
\chit Q(h) \chio = \chit A(x,hD_x)  \, \big( hD_{x_n} - i B(x,h D_x')  \big) \chio +   R(h),
$$
\begin{eqnarray} \label{opfactorization}
\|  R(h)  \|_{L^2\to L^2} = {\mathcal O}(h^{\infty}). 
 \end{eqnarray}

To achieve (\ref{opfactorization}), our ansatz is to use  the factorization at the level of principal symbols
\begin{eqnarray} \label{symbolbounds}
q_0(x,\xi)  = a_0(x,\xi)  ( \xi_n - i   B(x,\xi')), 
\end{eqnarray}
 to iteratively construct an $h$-smooth symbol $a \in S^{0}_h(T^*\Uhtau)$   satisfying  
\begin{equation}\label{def of A} 
q(x,\xi,h) \sim a(x,\xi,h)  \, \#  \Big( \xi_n - i  B(x,\xi') \Big), 
\end{equation}
where $q \sim \sum_{j=0}^{\infty} q_{-j} h^{j}$ is the total symbol of $Q(h).$ 
Since the elliptic symbol $\xi_n - i  B(x,\xi')$ is already in the desired form, we perturb the $a_0$-term only by adding lower order corrections in $h$ to match the total symbol of $q_0$. 
  The first term  $a_0$ already satisfies the desired equation in (\ref{def of A}):
  $$ q_0 = a_0 ( \xi_n - i  B ), \qquad a_0 \in {S^{0}(\Uhtau)}.$$

   For the second term $  a_{-1}$ one must solve an equation of the form 
  \begin{equation} \label{step0}
   q_{-1} = a_{-1} (\xi_n - i B ) + r_{-1}, \end{equation}
   where  $r_{-1} =- \partial_{\xi} a_0 \cdot \partial_{x} B:=-\sum_{j=1}^n\partial_{\xi_j} a_0  \, \partial_{x_j} B  \in S^{0}(\Uhtau)$. 
   Since $\xi_n - i  B \neq 0,$ we just solve for $a_{-1}$ in (\ref{step0}) and get
    \begin{equation} \label{step1}
 a_{-1} = (\xi_n - i  B)^{-1}  \, ( q_{-1} - r_{-1} ) = ( \xi_n  - i  B)^{-1} \, ( q_{-1} + \partial_{\xi} a_0    \partial_x  B ),
 \end{equation}
where we note that  $a_{-1} \in S^{0}(T^*\Uhtau).$ 
For the subsequent terms $a_{-m}$ with  $m \geq 2,$ we have 
\begin{equation} \label{step m}
a_{-m}:=  ( \xi_n - i  B)^{-1}{\Big( q_{-m} + i \sum_{1\leq \ell \leq m}  (-i)^\ell \sum_{|\alpha|=\ell}\frac{1}{\alpha !} ( \partial_{\xi}^{\alpha} a_{\ell-m} \, \partial_x^{\alpha}  B) } \Big). 
\end{equation}
It follows that  the decomposition (\ref{def of A}) holds for
$$
 a \sim \sum_{j=0}^{\infty} a_{-j} h^{j} \in   S_h^{0}(\Uhtau ). 
$$
Setting  $A(h):=a(x,hD_x)$   and noting $hD_{x_n} = Op_h(\xi_n)$,  it follows that the remainder $$ R(h):= \chit \,  \big( \,  Q(h)  - A(h) (hD_{x_n}  - i B(x, hD_{x'})  \, \big) \, \chio$$
satisfies $\| R(h) \|_{L^2 \to L^2} = O(h^{\infty}).$ 
 \end{proof}

\begin{rem} \label{lots of choices} Since $Q(h)$ is $h$-elliptic in the Fermi tube about $H$, the factorization in Proposition \ref{F and G} is by no means unique; indeed, one can factorize $Q$ over the Fermi tube in terms of {\em any} reference $h$-elliptic operator. As we will see in Proposition \ref{key prop}, the factorization corresponding to the specific choice of the reference  diffusion operator $hD_n - i B_0$, where $B_0>0$ is constant, is particularly convenient in converting the lower bound for $L^2$ eigenfunction mass $ \|u_h\|\sub{\mc{U}_H(\ep)}^2$  into an actual lower bound for the $L^2$ restrictions, $\|\gamma\sub{H} u_h\|_{L^2(H)}$  \end{rem}

\subsection{Exploiting the lacunary condition} \label{factor}


In this section, we explain how to combine Proposition \ref{F and G} with the lacunary condition on $Q(h)$ for the eigenfunction sequence $\{u_h\}$ to essentially allow us to work as if $(hD_{x_n}   - i B_0)u_h=O(e^{-C/h})$ on $\mc{U}\sub{H}(\ep)$, where $B_0$ is a positive constant.

In the following, it will be useful to define truncated cutoff functions. Given any cutoff $\chi \in C^{\infty}_0(\ep),$ we set
$$ \chi^{+}:= \chi \cdot {\bf 1}_{x_n \geq 0}, \qquad \mc{U}\sub{H}^+(\ep) := \{ x \in \mc{U}\sub{H}(\ep), \, x_n \geq 0 \}.$$

In  addition, given an open submanifold $\tilde{H} \subset H$ and  $\ep >0$ sufficiently small, we let $\psi(x') \in C^{\infty}_{0}(\tilde{H};[0,1]) $ with the property that there exists a proper open submanifold $\tilde{H}_\ep \subset \tilde{H}$ with {$\max_{x \in \tilde H}d(\tilde{H}_\ep, x) <\ep$} such that 
\begin{equation} \label{tangential}
 \psi |_{\tilde{H}_\ep}  =1.
  \end{equation}
  
Finally,  in the following,  $\gamma\sub{H}:M \to H$  denotes the restriction operator to $H$.

The proof of Theorem \ref{mainthm1} hinges on the following factorization result. \

\begin{prop} \label{key prop}  
 Let $\{u_h\}$ be a sequence of eigenfunctions satisfying \eqref{efn}, $H \subset M$  be a closed $C^\infty$-hypersurface, and suppose there exists an $H$-lacunary region for $\{u_h\}$ containing the Fermi tube $\mc{U}\sub{H}(2\ep)$.
Then, if  $B_0>0$ is any positive {\em constant}, $\tilde{H} \subset H$ is any open submanifold of $H$ and $\psi \in C^{\infty}_0(\tilde{H})$ is a tangential cutoff satisfying (\ref{tangential}), there exist operators $E(h):  C_0^{\infty}(\mc{U}\sub{H}(\ep)) \to C^{\infty}(\mc{U}\sub{H}^+(\ep))$ such that for some $C>0$,
\begin{equation}\label{bound}
 \big\|\chit^+  \,   \big(  hD_{x_n}   - i B_0 \big) \, \psi \big( I + E(h) \big) \, \chio  u_h  \big\|\sub{ L^2 } 
= O(e^{-C/h}),
\end{equation}
where 
$ \| \chit^+ E(h) \chio \|_{L^2 \to L^2} = O(h^\infty)$ and $\gamma\sub{H} E(h) = 0.$
\end{prop}

\begin{proof} 

We apply Proposition \ref{F and G} with $B(x,\xi'):=B_0,$  so  that
$B(x',h D_{x'})=  B_0  \, I $ is simply the multiplication operator by  $B_0 >0.$

Since the operator $A(h) \in \Psi^0_{h}(\mc{U}\sub{H}(2\ep) )$ from  Proposition \ref{F and G} is $h$-elliptic over $\mc{U}\sub{H}(2\ep)$, there exists a local parametrix $L(h) \in {\Psi_h^{0}(\mc{U}\sub{H}(2\ep))}$ such that 
\begin{equation} \label{parametrix}
 \|   \chio \big( \, L(h) A(h) - I \, \big)  \chio \|_{L^2 \to L^2} =  O(h^{\infty}). 
\end{equation}
Since $Q(h)$ is lacunary for $u_h$, there exists $C>0$, depending only on $\{u_h\}$, such that 
$$
\| \chit Q(h) \chio u_h \|_{L^2} = O(e^{-C/h}).
$$
 Therefore, since $  \| \chio [L(h),\chit] \|_{L^2 \to {L^2}}  = O(h^{\infty}),$  it follows from Proposition \ref{F and G}  (\ref{parametrix})  that 
%
\begin{eqnarray} \label{KEY}
 \chit  \big(  h D_{x_n} - i B(x,hD_{x'}) \big) \chio u_h =  \chit R'(h) \chio u_h + O(e^{-C/h}), \end{eqnarray}
where
\begin{equation} \label{R'}
 \| \chit R'(h) \chio \|_{L^2 \to L^2} = O(h^{\infty}).
\end{equation}


It follows from  (\ref{KEY})  that
\begin{equation} \label{keyprop1}
\chit^+  \psi \big(  h D_{x_n} - i B_0 \big) \chio u_h =  \chit^+ \psi R'(h) \chio u_h + O(e^{-C/h}).
\end{equation}

Moreover, by variation of constants, with $E(h):  C_0^{\infty}(\mc{U}\sub{H}(\ep)) \to C^{\infty}(\mc{U}\sub{H}^+(\ep))$ given by
\begin{equation} \label{keyprop2}
E(h) f (x',x_n) = -\frac{i}{h} \int_{0}^{x_n} e^{- (x_n - \tau) B_0 /h}  \, R'(h) f (x',\tau) \, d\tau, \quad x_n \in [0,\ep],
\end{equation} 
we obtain that $\gamma\sub{H} E(h) = 0$ and
\begin{equation} \label{keyprop3}
 \big(  hD_{x_n} - i B_0 \big) E(h) \chio =  -R'(h)\chio, \quad x_n \in [0,\ep].
 \end{equation}

 Thus,  from \eqref{keyprop1} it follows that
 \begin{equation} \label{keyprop2}
\chit^+  \psi \big(  h D_{x_n} - i B_0 \big) \chio u_h = - \chit^+ \psi \big(  hD_{x_n} - i B_0 \big) E(h) \chio u_h + O(e^{-C/h}).
\end{equation}
Since
\begin{equation} \label{commutator}
{[ hD_{x_n}-iB_0, \psi] =0,}
\end{equation}
 the bound in \eqref{bound} follows from \eqref{keyprop2}.
Also, by \eqref{R'} and the fact that $B_0>0$,
$$ \| E(h) \|_{L^2(\mc{U}\sub{H}(\ep)) \to L^2(\mc{U}\sub{H}^+(\ep))} 
= O(h^\infty).$$
\end{proof}

\begin{rem} \label{key idea}
 We note that  the final crucial step in the proof of  Proposition  \ref{key prop} involves showing that the error term $R'(h)$ in \eqref{R'} can also be factorized as in \eqref{keyprop1}. Proposition \ref{key prop}  implies that $(  hD_{x_n}   - i B_0 \big) v_h= O(e^{-C/h})$  with $v_h= \psi (I + E(h)) \chio u_h$, where we note that $v_h=\psi u_h$ on $H$ since $\gamma\sub{H}E(h)=0$. We will also use that the $L^2$-mass of $v_h$ is comparable to that of $\psi u_h$ since  $ \| \chit^+ E(h) \chio \|_{L^{\infty} \to L^{\infty}} = O(h^\infty)$.  \end{rem}

\begin{rem} \label{localization}
 A key step in the proof of Proposition \ref{key prop} that allows us to localize the eigenfunction restriction bounds
 to an open submanifold $\tilde{H} \subset H$ involves the commutator condition $[ hD_{x_n}-iB(x,hD_{x'}), \psi] =0$ in (\ref{commutator}) where
  $\psi = \psi(x') \in C^{\infty}_0(\tilde{H})$ is a tangential cutoff  satisfying (\ref{tangential}). Since trivially $[hD_{x_n},\psi] =0,$ (\ref{commutator}) is equivalent to $[B(x,hD'), \psi] = 0$ and the latter requirement forces us to choose the tangential $h$-psdo  to be  a {\em constant} multiplication operator; that is, 
  $B(x,hD') = B_0$ with $B_0>0.$
 \end{rem}

\subsection{Proof of Theorem \ref{mainthm1}}  \label{proof}

Let $\tilde{H} \subset H$  be an open submanifold and  choose $q_0 \in K=\supp\!(\pi_* \mu)$  and  $q\sub{H} \in \tilde H$ so that
$$
0<d(q_0, q\sub{H})=d(K,\tilde H)<  \tau_0,
$$
where $\tau_0$ is as in Theorem \ref{mainthm2}.


In the following we  let $(x',x_n)$ be Fermi coordinates adapted to $H$,
$$
 H=\{x_n=0\},\qquad q\sub{H}=(0,0),
$$
and we assume they are well defined for $(x',x_n)\in \mc{U}\sub{H}(2 \ep).$

We continue to let $\chi_j \in C^{\infty}_0(\mc{U}\sub{H}(2\ep))$, for  $j=1,2, $ be the nested cutoff functions in Section \ref{factor}.
In general, for each $0\leq \tau < 2\ep$ we define the level hypersurface 
$$
H_{\tau}:= \{(x',x_n):\; x_n = \tau\}, \qquad H_0=H.
$$

We note that for $0<x_n<2\ep$  there is a natural diffeomorphism 
$\kappa_\tau: H \to H_{x_n}$ that in Fermi coordinates takes the form $\kappa_\tau (x') = (x',\tau).$
Consequently, using $\kappa_\tau$ to parametrize $H_{x_n}$ by $H$ together with the fact that $({\kappa_\tau})_*(d\sigma\sub{\!H})=d\sigma\sub{\!H_\tau}$, for every $v \in L^2(\mc{U}\sub{H}(2\ep))$
\begin{equation} \label{mon1}
 \| \gamma\sub{H_\tau} v   \|\sub{L^2(H_\tau)}^2 = \int_{H} |v(x',\tau)|^2 \, d\sigma\sub{H}(x') = \| \kappa_\tau^* \gamma\sub{H_\tau} v \|^2\sub{L^2(H)}. 
 \end{equation}

\begin{figure}
\caption{}
\includegraphics{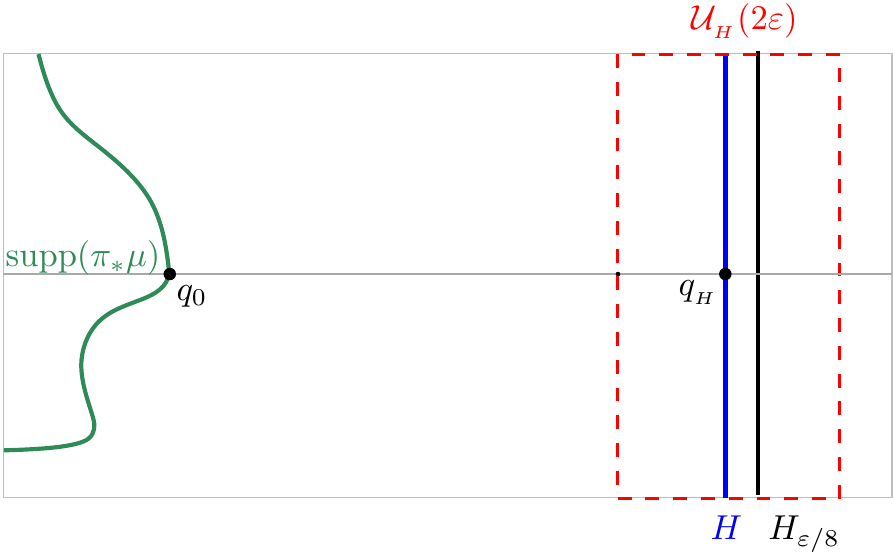}
\end{figure}

Let $\psi \in C^{\infty}_0(\tilde{H})$ be a tangential cutoff satisfying (\ref{tangential}) and $E(h):  C_0^{\infty}(\mc{U}\sub{H}(\ep)) \to C^{\infty}(\mc{U}\sub{H}^+(\ep))$ as in Proposition \ref{key prop}.
Set
$$ v_h:= \psi (I + E(h)) \chio u_h, \qquad  {on}\;\; \mc{U}\sub{H}^+(2\ep) \cap \text{supp} \, \psi,$$
where we note that since $\gamma\sub{H} \psi E(h) \chio u = \psi \gamma\sub{H} E(h) \chio u =0$ {and $\chio |_{H} =1$,}  it follows that 
$$\gamma\sub{H} v_h = \gamma\sub{H} \psi \chio u_h = \psi \gamma\sub{H} \chio u_h = \psi \gamma\sub{H} u_h.$$

Note that from Proposition \ref{key prop}, there is $C>0$ depending only on $\{u_h\}$ such that
$$
 \big\|\chit^+  \,   \big(  hD_{x_n}   - i B_0 \big) v_h  \big\|\sub{{L^2}} 
= O(e^{-C/h}).
$$
Since   $\chit^+ (x_n) = 1$  for $x_n \in [0, \frac{\ep}{8}]$, it follows that
\begin{equation} \label{evolution0}
\tfrac{1}{2}h \partial_{x_n} \int_{H} |v_h(x',x_n)|^2 d\sigma\sub{H}(x') = -  \int_{H} B_0 v_h(x',x_n) \overline{v_h(x',x_n)} \, d\sigma\sub{H}(x') + O(e^{-\tilde{C}/h}).
 \end{equation}
In view of (\ref{mon1}) one can rewrite (\ref{evolution0}) in the form
\begin{equation} \label{evolution}
\tfrac{1}{2}h \partial_{x_n} \, \| \gamma\sub{H_{x_n}} v_h \|_{L^2(H_{x_n})}^2 =  - B_0 \| \gamma\sub{H_{x_n}} v_h \|^2_{L^2(H_{x_n})} + O(e^{-\tilde{C}/h}), \qquad  x_n \in [0,\tfrac{\ep}{8}].
 \end{equation}

Integration of \eqref{evolution} over  $  0 \leq x_n  \leq \frac{\ep}{8}$ and multiplication by $-1$ gives
\begin{align} \label{monupshot0}
{h} \| \gamma\sub{H} v_h \|\sub{L^2(H) }^2 -   {h} \| \gamma\sub{H_{ \ep/8 } } v_h \|\sub{L^2(H_{\ep/8}) }^2  =  2 B_0  \big \|    v_h \big \|\sub{L^2(\mc{U}\sub{H}(\ep/8) )}^2  + O(e^{-\tilde{C}/h}), 
\end{align}
and consequently,
\begin{align} \label{monupshot1}
{h} \| \gamma\sub{H} v_h \|\sub{L^2(H) }^2    \geq  2 B_0  \big \|    v_h \big \|\sub{L^2(\mc{U}\sub{H}(\ep/8) )}^2  + O(e^{-\tilde{C}/h}). 
\end{align}
 
We also note that, by Proposition \ref{key prop}, we have $ \| \chit^+ E(h) \chio \|_{L^2 \to L^2} = O(h^\infty)$ and so,
$$ 
\big \|    v_h \big \|\sub{L^2(\mc{U}\sub{H}(\ep/8) )} =  ( 1 + O(h^{\infty}) ) \, \big \|   \psi  u_h \big \|\sub{L^2(\mc{U}\sub{H}(\ep/8) )} \geq   \tfrac{1}{2} \big \|   \psi  u_h \big \|\sub{L^2(\mc{U}\sub{H}(\ep/8) )}.
$$

Since $\psi |_{ \tilde{H}_\ep} =1$ on the open submanifold $\tilde{H}_{\ep} \subset \tilde{H}$ with {$\max_{x \in \tilde H}d(\tilde{H}_\ep, x) <\ep$ (see (\ref{tangential})),} it follows that
 \begin{equation} \label{lb}
\big \|    v_h \big \|\sub{L^2(\mc{U}\sub{H}(\ep/8) )}^2 \geq \tfrac{1}{4} \, \big \|    u_h \big \|\sub{L^2(\mc{U}\sub{\tilde H_\ep}(\ep/8) )}^2.
\end{equation}

 We next find a lower bound for the RHS of \eqref{lb} by applying  Theorem \ref{mainthm2}. Indeed, Theorem \ref{mainthm2} yields that for $\ep >0$ arbitrarily small, we  have

%
\begin{equation}\label{quc}
  \big \|    v_h \big \|\sub{L^2(\mc{U}\sub{H}(\ep/8) )}^2 
 \geq  \tfrac{1}{2} \, \big \|    u_h \big \|\sub{L^2(\mc{U}\sub{\tilde H_\ep}(\ep/8) )}^2
\geq C_\ep  e^{- { 2 \, ( d(\tilde{H_\ep},K) + \epsilon) /}{h}} \geq C_\ep e^{- { 2 \, ( d(\tilde{H},K) + 2 \epsilon) /}{h}}. 
 \end{equation} 
  
  In the last estimate in (\ref{quc}), we use (\ref{tangential}) and the fact that $\max_{x \in \tilde H}d(\tilde{H}_\ep, x) <\ep$ where $\ep>0$ is arbitrarily small but fixed independent of $h$.

Combining \eqref{monupshot1} and \eqref{quc}, and recalling that $\gamma\sub{H} v_h = \psi \gamma\sub{H} u_h$, implies that for any $\ep >0$ and $h \in (0,h_0(\ep)]$ there are constants $C_\epsilon >0$ and $C_{\epsilon}' >0$ such that

\begin{align} \label{monupshot5}
 h \| \psi \gamma\sub{H} u_h \|\sub{L^2(H) }^2  
  &\geq  C_\epsilon e^{ {  -2 ( d(\tilde{H},K) + 2 \epsilon)/}{h}}  - C_{\ep}' e^{-\tilde{C}/h}.
 \end{align}

To complete the proof of Theorem \ref{mainthm1}, we note that, since the second term on the RHS of (\ref{monupshot5}) depends only on the eigenfunction sequence (and not on ${\tilde{H}}$), it is clear that it can be absorbed in the first term provided one chooses $\tilde{H}$ sufficiently close to $K$, with $2  \, d(\tilde{H},K) < \tilde{C}.$ Thus, for such $\tilde{H}$ it follows from (\ref{monupshot5}), and the fact that $\psi \in C^{\infty}_0(\tilde{H})$, that
$$ h \int_{\tilde{H}} |u_h|^2 \, d\sigma\sub{\tilde H} \geq  C_\epsilon' e^{ {  -2 ( d(\tilde{H},K) + 2\epsilon)/}{h}}.$$
Since $\ep>0$ is arbitrarily small, this concludes the proof of  Theorem \ref{mainthm1}.
\qed

  \section{The case of Schr\"{o}dinger operators} \label{schrodinger}

Let $(M,g)$ be a compact $C^\infty$  Riemannian manifold,  $V \in C^{\infty}(M,\R)$.
Consider  the classical Schr\"{o}dinger operator 
$$
P(h) = -h^2 \Delta_g + V - E,
$$
 where  $E$ is a regular  value for $V.$ 
 In the classically forbidden region $\{ V > E \},$ the eigenfunctions $u_h$  satisfy the Agmon-Lithner estimates \cite{Zw}: for all $\delta>0$ there is $C(\delta)>0$ such that
 \begin{equation} \label{agmon}
  |u_h(x)| \leq C(\delta) e^{- [ d_E(x) - \delta ] /h}, \qquad  x \in \{ V >E \} ,
  \end{equation}
 where $d_E(x)$ is the distance from $x$ to $\{ V = E \}$ in the Agmon metric $g_E = (V-E)_{+} |d x|^2.$
  As a immediate consequence of (\ref{agmon}), it follows that  if $\mu$ is a defect measure associated to a sequence  $\{u_h\}$ of $L^2$-normalized  Schr\"odinger eigenfunctions, $P(h) u_h=0$, then its support is localized in the allowable region; that is,
\begin{equation}\label{allowed mu}
\supp\!(\pi_* \mu) \subset \{x\in M:\; V(x) \leq  E \}.
\end{equation}
We show that if $H$ lies inside the forbidden region $\{V>E\}$ but it is such that a Fermi neighborhood of it reaches the support $\supp\!(\pi_* \mu) $, then $H$ is a \emph{good} curve for $\{u_h\}$ in the sense of \eqref{Goodness}.

The proof of Theorem \ref{mainthm3} follows the same outline as in the homogeneous case in Theorem \ref{mainthm1}. Here, we explain the relatively minor changes required to prove the analogue of the Carleman estimates in Theorem \ref{mainthm2} and refer to the previous sections for further details.

\begin{theo} \label{mainthm4}
Let $(M,g)$ be a compact $C^\infty$ Riemannian manifold.
Let $\{u_h\}$ be a sequence of eigenfunctions satisfying \eqref{efnSchr} and let $\mu$ be a defect measure associated to it. 
Let $H \subset M$  be a $C^\infty$-hypersurface (possibly with boundary $\partial H$) and suppose there exist $q_0 \in \supp\!(\pi_* \mu)$  and $q\sub{H} \in H \backslash \partial H$ such that $d(q_0,q\sub{H})=d(q_0,H)$.

For all $ \beta > \max_{x \in M} | V(x) - E|^{1/2}$ there exists  $\tau_0>0$ such that if 
$$
0<d(q_0,H)< \tau_0,
$$
then   there are constants $C_0>0$ and $h_0>0$ such that 
$$ \| u_h \|\sub{L^2(\mathcal{U}\sub{H}^\epsilon)}  \geq C_0 e^{ - \beta d(q_0, H) /h}$$
for all $h \in (0,h_0].$ 
\end{theo}

\begin{proof}

To prove the Carleman analogue of Theorem \ref{mainthm2}, we need to adapt the argument slightly by constructing a modified weight function.  
Let  
$$ \beta^2 > \max_{y \in M} | V(y) - E |,  \qquad \delta_0:= \beta^2 - \max_{y \in M} | V(y) - E |.
$$

Given a geodesic sphere $Y_{q,r}:=\partial B(q, r)$, let $a\sub{\,Y_{q,r}}$ be the quadratic form dual to the induced metric on $Y_{q,r}$ and $b\sub{\,Y_{q,r}}$ be the quadratic form dual to the second fundamental form for $Y_{q,r}$. In particular,  for all $y' \in Y_{q,r}$, the eigenvalues of $b\sub{\,Y_{q,r}}(y')$ with respect to $a\sub{\,Y_{q,r}}(y')$ are the principal curvatures of $Y_{q,r}$ and are are strictly positive. Since the principal curvature of $Y_{q,r}$ grows to infinity as $r\to 0^+$, there exists  $r_0>0$ such that 
\begin{equation}\label{e:curv}
\min_{q\in M}\min_{r\leq r_0}  \min_{y' \in Y_{q,r}} \{2b\sub{\,Y_{q,r}}(y', \xi'):  \; a\sub{\,Y_{q,r}}(y',\xi') \geq \tfrac{1}{2}\delta_0\}> \max_{y\in M}|\nabla V(y)|.
\end{equation}
Condition \eqref{e:curv} implies that the principal curvatures of any geodesic sphere $Y_{q,r}$ with  $0<r\leq r_0$ are bounded below by $\tfrac{1}{2}\max_{y\in M}|\nabla V(y)|$. This will be used to prove that $\psi$ as defined  below is a Carleman weight.

Next, let $q_0 \in \supp\!(\pi_* \mu)$. Let $r_{q_0}$  be the maximal radius  for which the exponential map $\exp_{q_0}$ is a diffeomorphism on $B(q_0, 2r_{q_0})$. Let $s_{q_0}=\min\{r_{q_0}, r_0\}$ and work with $\tau_0<\tfrac{1}{4}s_{q_0}$ to be chosen later.
Then,    let $H \subset M$  be a $C^\infty$-hypersurface and   $q\sub{H} \in H \backslash \partial H$ such that
\begin{equation}\label{distance2}
\tau\sub{H}:=d(q_0, q\sub{H})=d(q_0,H), \qquad 0<\tau\sub{H} < \tau_0.
\end{equation}

Let $\gamma$ be a unit speed geodesic joining $q_0=\gamma(0)$ with $q\sub{H}=\gamma(\tau\sub{H})$.  Let $q=\gamma(s_{q_0})$  and set 
$
 Y:=\partial B(q, s_{q_0}).
$
As before, we work with $(y', y_n)$ being geodesic normal coordinates adapted to $Y$, in which
$
 Y=\{y_n=0\},
 $
 and
 $ 
 q_0=(0,0),
$
 and with $\{y_n>0\}$ corresponding to points in the interior of $B(q, s_{q_0})$.  
 
 Note that these coordinates are well defined for $|y_n|<2\tau\sub{Y}:=s_{q_0}$ and $|y'|<c\sub{Y}$ for some $c\sub{Y}>0$, since $\exp_q$ is a diffeomorphism on $B(q, 2r_{q_0})\supset B(q, 2s_{q_0})$. Furthermore, since $\tau_0<\tfrac{1}{4}s_{q_0}$, by \eqref{distance2} we also have $\tau\sub{H}<\tau_0 <\tau\sub{Y}$.
In particular,  with $\ep\sub{Y}$ as in \eqref{constconditions}, the Fermi coordinates with respect to $Y$ are well defined on  $\mc{W}\sub{ Y}(\tau\sub{H}, \ep\sub{Y})$ as in \eqref{W}.

In analogy with \eqref{weight} and \eqref{v_h}, for $\epsilon >0$  we set
\begin{equation} \label{Sweight}
\psi(y',y_n) :=  \beta y_n +  2 \tau\sub{H} \rho_{\ep}(y'). 
\end{equation}
Note that, as in \eqref{ppal-symbol},
\begin{equation}\label{ppal-symbol-schrodinger}
 p(y, \xi)=\xi_n^2+a\sub{\,Y}(y', \xi') -2y_n b\sub{\,Y}(y', \xi') +  R(y, \xi') + V(y) - E,
 \end{equation}
with $R(y, \xi')=O(y_n^2 |\xi'|^2)$.
Next, note that  provided $|y_n| \leq \tau_0,$
 \begin{align} \label{sch1}
& \Re p_{\psi} = \xi_n^2 + a\sub{\,Y}(y',\xi') - 2 y_n b\sub{\,Y}(y',\xi') + V(y) - E  - \beta^2 + O(\tau_0^2 |\xi'|^2), \nonumber \\
& \Im p_{\psi} = 2 \beta \xi_n + O(\tau_0 |\xi'|). \hspace{2.8in}
 \end{align}
Therefore, 
\begin{align*}
 \{p_{\psi}=0\}
\!\!=\!\!  \big\{ (y,\xi) \in T^* (\mc{W}\sub{ Y}(\tau\sub{H}, \ep\sub{Y})):  a\sub{\,Y}(y', \xi') = \beta^2+E-\!V(y)+ O(\tau_0), \, \xi_n = O(\tau_0)\big\},
\end{align*}
and so
\begin{align} \label{comm}
\{ \Re p_{\psi}, \Im p_{\psi} \} (y,\xi) 
&=  4 {\beta}b\sub{\,Y}(y', \xi') {-2\beta}\partial_{y_n}V(y)  + O(\tau_0), \quad  (y,\xi) \in \{p_{\psi}=0\}.
\end{align}
 
 Next, let $\tau_0$ be small enough so that on  $\{p_{\psi}=0\}$ we have $a\sub{\,Y}(y', \xi')  \geq \tfrac{1}{2}\delta_0$. Then,  the lower bound in  \eqref{e:curv} together with \eqref{comm}
yield
 $\{ \Re p_{\psi}, \Im p_{\psi} \} (y,\xi) >0$ on  $\{p_{\psi}=0\}$.
 This shows that $\psi$ is a Carleman weight on $\mc{W}\sub{ Y}(\tau\sub{H}, \ep\sub{Y})$.
 
One then proceeds exactly as in the proof of Theorem \ref{mainthm2} to show that
\begin{equation} \label{carleman/schrodinger}
\| u_h \|_{L^2(\mathcal{U}_H(\ep))} \geq C(\ep) e^{- ( \beta d(q_0,H) + \ep)/h}. 
\end{equation}

 \end{proof}
The proof of Theorem \ref{mainthm4} then follows exactly as for Theorem \ref{mainthm2} after noting the following.
Let $\{u_h\}$ be a sequence of $L^2$-normalized eigenfunctions of a Schr\"odinger operator $P(h) = - h^2 \Delta_g + V-E $.
Choose $0<\tau_1<\tau_0$ such that  $$\Uho\subset \{x\in M:\; V(x) > E \}.$$
We recall that in this case, by \eqref{allowed mu},  $K=\supp(\pi_*\mu) \subset \{ V \leq E \}.$
Since $P(h)$ is elliptic on $\Uho$, it has a left parametrix $L(h)$. Thus,
$$
Q(h):=L(h)P(h) \in \Psi^0_h(\Uho)
$$ is  $h$-elliptic over the set $\{ V >E \}$ and  $Q(h) u_h = 0$ {since the $u_h$ are eigenfunctions.}  We conclude from Remark \ref{differential} that $Q(h)$ is a lacunary operator for $\{u_h\}$.

\section{Examples} \label{examples}
In this section we present several examples to which our results apply.

\subsection{Warped products}
Let $(M,g\sub{M})$ and $(N, g\sub{N})$ be two compact $C^\infty$ Riemannian manifolds.  We work on the warp product manifold $M \times_{f} N$ endowed with the metric $g=g\sub{M} \oplus f^2 g\sub{N}$, for some function $f \in C^\infty(M, \R\backslash\{0\})$. 

Let $\{\varphi_h\}_h  \in C^\infty(N)$ be a sequence of normalized eigenfunctions
\begin{equation}\label{efn N}
-h^2 \Delta_{g\sub{N}}\varphi_{h}=\varphi_{h}, \qquad \|\varphi_{h}\|\sub{L^2(N)}=1,
\end{equation}
and for each $\varphi_h$ consider the subspace 
$$
\mathcal F_h = \{  v\otimes  \varphi_{h} : \; v \in L^2(M)\} \subset L^2(M \times_{f} N).
$$
Since $g=g\sub{M} \oplus f^2 g\sub{N}$, with  $V:=f^{-2}  >0$ we have 
$$
-h^2\Delta_g=-h^2 \Delta_{g_M}  - V h^2 \Delta_{g\sub{N}} + hL(h),
$$
 where 
 $$
W(h)=-{n }f^{-1}\,h\nabla_{g_M}  f, \qquad \qquad n=\dim N.
 $$ 
 Note that $W(h)$  is a first order  differential operator on   $L^2(M \times N)$ which acts by differentiating in the $M$ variables only.
 
In particular,  $\mathcal F_h  \subset L^2(M \times_{f} N)$ is invariant under $-h^2\Delta_g$ and 
$$
P(h) :=-h^2\Delta_g \big|_{\mathcal F_h }= -h^2 \Delta_{g_M}  +V + hW(h).
$$
Using that  $-h^2\Delta_g $ is self-adjoint on $L^2(M \times_{f} N)$, it is immediate to see that  $P(h)$ is self-adjoint when viewed as an operator acting on $(M, \langle \cdot \,,\, \cdot \rangle_{\tilde g_M})$ where $\langle  v_1 , v_2 \rangle_{\tilde g_M}=\int_M   v_1 \,\overline{v_2} \, f^n dv_{g_M}$.

\begin{lem}
Let $E$ be a regular value for $V$ and let  $\{v_h\}$ be a sequence of eigenfunctions,
$(P(h)-E)v_h=0$, with defect measure $\mu$. Let $u_h=v_h\otimes  \varphi_{h}$ be the sequence of eigenfunctions $(-h^2\Delta_g -E)u_h=0$ with $\varphi_h$ as in \eqref{efn N}.
 
Let $H \subset \{x\in M: \,V(x)>E\}$ be a closed $C^\infty$ hypersurface. Then, there exists  $\tau_0>0$ such that the following holds. If there exists  $q_0 \in \supp\!(\pi_* \mu)$  with
$
0<d(q_0,H)< \tau_0,
$
then  for all $\ep >0$  there are $C_0(\ep)>0$ and $h_0(\ep)>0$ such that 
$$ \| u_h \|\sub{L^2(H \times N)} \geq C_0 e^{ -(d(q_0, H)+ \ep)/h},$$
for all  $h \in (0,h_0(\ep)].$
\end{lem}

\begin{proof}
The operator $Q(h)=P(h)^{-1}(P(h)-E)  \in \Psi_h^0(M)$ acts on $M$ and
\begin{equation}\label{e:zero}
Q(h)v_h=0.
\end{equation}

Let $(x',x_n)$ be Fermi coordinates on 
$M$ adapted to $H=\{x_n=0\}$. Then,
$Q(h)=Op_h(q)$ with 
$$
q_0(x,\xi)=(\xi_n^2+r(x, \xi') +V(x))^{-1}(\xi_n^2+r(x, \xi') +V(x)-E).
$$
It follows that $Q(h)$ is $h$-elliptic, and hence Remark \ref{differential} and \eqref{e:zero} yield that $Q(h)$ is a lacunary operator for $\{v_h\}$ in a Fermi neighborhood of $H \subset \{x\in M: \,V(x)>E\}$.
The result then follows from Theorem \ref{mainthm3} and the fact that since $M$ is compact there exists $C>0$ such that $f\geq C$ and so 
$\| u_h \|\sub{L^2(H \times N)} \geq  C \|v_h\|\sub{L^2(H)}$.
\end{proof}

\subsection{Eigenfunctions of quantum completely integrable (QCI) systems}

 Let $(M,g)$ be a compact $C^\infty$  Riemannian manifold of dimension $n$ and
 let $\{P_{j}(h)\}_{ j =1}^n$ be a  QCI system of $n$ real-smooth, self-adjoint $h$-partial differential operators with 
 $$
 [P_i(h), P_j(h)] =0, \qquad  i \neq j,
 $$
  and  such that $\sum_{j=1}^n P_j(h)^* P_j(h)$ is  $h$-elliptic with left parametrix $L(h)$. 
  We apply our results to studying restrictions of appropriate subsequences $\{u_{h}\}$ of joint eigenfunctions of the $P_j(h)$ for $j=1,\dots, n$.
  Examples include joint eigenfunctions on spheres and tori of revolution, eigenfunctions on hyperellipsoids with distinct axes, eigenfunctions of Neumann oscillators, Lagrange and Kowalevsky tops and spherical pendulum (see \cite{HW} for further examples).

Without loss of generality, we assume that $P_{j} \in \Psi_{h}^{2}(M)$ for $j=1,..,n.$
and also assume that 
  $$P_1(h) = -h^2 \Delta_g,  \quad \text{or} \quad  P_1(h) = - h^2 \Delta_g +V.$$
   All QCI systems on compact manifolds that we are aware of satisfy these properties.

   Let $${\mathcal P} = (p_1,...,p_n): T^*M \to \R^n$$
    be the associated moment map where $p_j = \sigma(P_j(h))$, and suppose $E= (E_1,...,E_n) \in  {\mathcal P}(T^*M)$ is a regular value of the moment map.  By Liouville-Arnold, the level set
$$ \Lambda\sub{E}:= \{ (x,\xi) \in T^*M: \,\, p_{j}(x,\xi) = E_j, \,\, j=1,...,n \}$$
is a finite union of $\R$-Lagrangian tori. To simplify the writing somewhat we assume here that $\Lambda_E$ is connected.  Let $\pi: T^*M \to M$ be the canonical projection and $\pi_{\Lambda_E}$ be its restriction to $ \Lambda\sub{E}$.

Let $u_{E,h} \in C^{\infty}(M)$ be joint eigenfunctions of the $P_j(h)$'s with joint eigenvalues $E_j(h) = E_j + o(1).$ Then, since 
$$
Q(h) u_{h,E} = 0, \qquad Q(h):=L(h)W\sub{E}(h) \in \Psi^0_h(M),
$$
$$
 W\sub{E}(h):=\sum_{j=1}^n (P_j(h)-E_j(h))^* ( P_j(h) - E_j(h)),
$$
 it follows that if $\mu$ is a defect measure for $\{u_{h,E}\}$, then it concentrates on the torus $\Lambda\sub{E}$. Indeed, it  follows from the quantum Birkhoff normal form expansion for $Q(h)$ near $\Lambda_E$  \cite{TZ03} that 
  $$
  \mu = (2\pi)^{-n} | d\theta_1 \cdots d \theta_n | 
  $$
  where $\theta$'s are the angle variables on the tori $\Lambda_{E}$ and so,
\begin{equation}\label{defect qci}
 K = \supp (\pi_* \mu ) = \pi (\Lambda\sub{E}).
\end{equation}

Let $\tilde{K} \Supset K$ with a closed hypersurface  $H \subset (M\setminus \tilde{K})$ that is is sufficiently close to $K$. Then, if 
  $\Uho\subset  M\setminus \tilde{K}$,  it is not difficult to show that (\cite{GT20} Lemma 3.5)
$$ \sigma( W\sub{E}(h))(x,\xi) = \sum_{j=1}^n (p_j(x,\xi) - E_j)^2 \geq C \langle \xi \rangle ^4, \qquad (x,\xi) \in T^*(M \setminus \tilde{K}),$$
and so, since   $C'' \langle \xi \rangle^4 \leq \sum_j  |p_j(x,\xi)|^2 \leq C' \langle \xi \rangle^4,$ it follows that $Q(h) \in \Psi^0_h(\Uho)$ is  $h$-elliptic. Also,  since $Q(h) u_{E,h} = 0,$ the operator $Q(h)$ is lacunary for the subsequence $\{ u_{E,h} \}$ in the Fermi tube $\Uho.$


An application of  
 Theorem \ref{mainthm1}  (resp. Theorem  \ref{mainthm3}) in the case where $P_1(h)$ is a Laplacian (resp. Schr\"{o}dinger operator) yields the following result.
\begin{theo}\label{qci}
Let $\{u_{E,h}\}$ be a sequence of joint eigenfunctions of the $P_j(h)$'s with joint eigenvalues $E_j(h) = E_j + o(1)$, and  let $\mu$ be the associated defect measure.  There exists  $\tau_0>0$ such that if
 $H \subset M \setminus \pi(\Lambda\sub{E})$   is a closed $C^\infty$ hypersurface with
 $$
 d(H, \pi(\Lambda\sub{E})) < \tau_0,
$$ 
then for any $\ep >0$ there are constants $C_0(\ep)>0$ and $h_0(\ep)>0$ such that 
$$ \| u_h \|\sub{L^2(H)} \geq C_0(\ep)  e^{ - [ d(H, \pi(\Lambda\sub{E})) + \ep] /h},$$
for all  $h \in (0,h_0(\ep)].$ 
\end{theo}

 \end{document}